\documentclass[10pt]{article}
\usepackage[utf8]{inputenc}
\usepackage[english]{babel}
\usepackage{mathptmx}       
\usepackage{helvet}         
\usepackage{courier}        
\usepackage{type1cm}        
\usepackage{makeidx}         
\usepackage{graphicx}        

\usepackage[bottom]{footmisc}

\usepackage{amsmath}
\usepackage{amssymb}
\usepackage{amsfonts}
\usepackage{amsbsy}
\usepackage{amscd}
\usepackage{amstext}
\usepackage[english]{babel}
\usepackage{chapterbib}
\usepackage{color}
\usepackage{graphics}
\usepackage{epsfig}
\usepackage{subfigure}
\usepackage{wrapfig}
\usepackage{psfrag}
\usepackage{url}
 \usepackage{bm}

\DeclareMathOperator{\Lapl}{\triangle}

\title{A Discontinuous Coarse Spaces (DCS) Algorithm for Cell Centered
  Finite Volumes based Domain Decomposition Methods: the
  DCS-RJMin algorithm}
\author{Kévin Santugini\thanks{Universit\'e Bordeaux, IMB, CNRS UMR5251,  MC2, INRIA Bordeaux - Sud-Ouest  \texttt{Kevin.Santugini\@math.u-bordeaux1.fr}}}

\usepackage{amsthm}
\newtheorem{theorem}{Theorem}[section]
\newtheorem{proposition}[theorem]{Proposition}
\theoremstyle{definition}
\newtheorem{algorithm}[theorem]{Algorithm}

\usepackage{fancyhdr}

\fancypagestyle{firstpage}{
\fancyhf{}

}

\newcommand{\D}{\mathrm{D}}

\begin{document}

\maketitle
\begin{abstract}
In this paper, we introduce a new coarse space algorithm, the ``Discontinuous
Coarse Space Robin Jump Minimizer'' (DCS-RJMin),
to be used in conjunction with one-level domain decomposition methods
(DDM). This new algorithm makes use of 
Discontinuous Coarse Spaces(DCS), and is designed for DDM that
naturally produce discontinuous iterates such as Optimized Schwarz
Methods(OSM). This algorithm is suitable both at the
continuous level and for cell-centered finite volume
discretizations.  At the continuous level, we prove, under some conditions on the
parameters of the algorithm, that the difference between two
consecutive iterates goes to $0$. We also provide numerical 
results illustrating the convergence behavior of the DCS-RJMin algorithm.
\end{abstract}

\section{Introduction}
Due to the ever increasing parallelism in modern computers, and the
ever increasing affordability of massively parallel calculators, it is
of utmost importance to develop algorithms that are not only parallel
but scalable. 
In this paper, we are interested in
Domain Decomposition Methods(DDM) which are one way to parallelize the numerical
resolution of Partial Differential Equations(PDE). 

In Domain Decomposition Methods, the whole domain is subdivided in
several subdomains and a computation unit is assigned to each
subdomain. In this paper, we only consider non-overlapping domain decompositions.
The numerical solution is then computed in parallel inside each subdomain
with artificial boundary conditions. Then, subdomains exchange
information between each other. This process is reapplied until convergence. In
practice, such a scheme, called iterative DDM, should be accelerated using
Krylov methods. However, for the purpose of analyzing an algorithm, 
it can be interesting to work directly with the iterative
algorithm itself as Krylov acceleration is so efficient it can hide
away small design problems in the algorithm. 

In one-level DDM, only
neighboring subdomains exchange information. Most classical DDM are
one-level. While one-level DDM can be very efficient and converge in a
few iterations, they are not scalable: convergence can never occur
before information has propagated between the two furthest apart
subdomains. \textit{I.E.}, a one level DDM must iterate at least as
many times as the diameter of the connectivity graph of the domain decomposition. Typically, if
$N$ is the number of subdomains, this means at least $O(N)$ iterations
for one-dimensional problems, $O(\sqrt{N})$ for two-dimensional
ones and $O(\sqrt[3]{N})$ for three-dimensional ones.
For DDM to be scalable, some kind of global information exchange   is
needed. The traditional approach to achieve such global information
exchange is adding a coarse space to a pre-existing one-level DDM.

To the author knowledge, the first use of
coarse spaces in Domain Decomposition Methods can be found 
in~\cite{Nicolaides:1987:DeflationConjugateGradientApplicationBoundaryValueProblems}.
Because coarse spaces enable global information exchange, scalability
becomes possible. Well known methods with coarse spaces are the two-level Additive Schwarz
method~\cite{Dryja:1987:AVS}, the FETI method~\cite{Mandel:1996:BDP},
and the balancing Neumann-Neumann 
methods~\cite{Mandel:1993:BalancingDomainDecomposition,Dryja:1995:SMN,Mandel:1996:CSM}.
Coarse spaces are also an active area of research, see
\cite{Dolean.Nataf.Scheichl.Spillane:AnalysisTwo-LevelSchwarzMethods,Nataf:2011:CSC}
for high contrast problems. It is not trivial to add an effective
coarse space to one-level DDM that produce discontinuous iterates such
as Optimized Schwarz Methods,
see~\cite{Dubois:2009:CBO,Dubois:2012:TOS},
and~\cite[chap.5]{Dubois:2007:OSM}.

In~\cite{Gander.Halpern.Santugini:2013:DD21-DCSDMNV}, the authors introduced the idea
of using discontinuous coarse spaces. Since many DDM algorithms produce discontinuous iterates, the use of 
discontinuous coarse corrections is needed to correct the discontinuities
between subdomains. In that proceeding, one
possible algorithm, the DCS-DMNV
(Discontinuous Coarse Space Dirichlet Minimizer Neumann Variational),
was described at the continuous level and at the discrete level for
Finite Element Methods on a non-overlapping Domain Decomposition.
In~\cite{Santugini:2014:DCS-DGLC}, a similar method, the DCS-DGLC algorithm
 was proposed. Both the DCS-DMNV and the DCS-DGLC are well suited to
 finite element discretizations. Also, a similar approach was proposed
 in~\cite{Gander.Halpern.Santugini:2013:ANC}
for Restricted Additive Schwarz(RAS), an overlapping DDM, 

It was proven recently that the proof of convergence for 
Schwarz found in~\cite{Lions:1990:SAM,Despres:1991:DomainDecompositionHelmholtzProblem} can be extended to the 
Discrete Optimized Schwarz algorithm 
with cell centered finite volume methods,
see~\cite{Gander.Kwok.Santugini:EnergyProofDiscreteOSMCellCenteredFV}. It would be interesting to have a
discontinuous corse space algorithm that is suited to cell centered
finite volumes. Unfortunately, neither the DCS-DMNV algorithm nor the 
DCS-DGLC algorithm are practical for cell centered-finite volume
methods: the stiffness matrix necessary to compute the coarse correction 
isn't as sparse as one would intuitively believe. In this paper, our
main goal is to describe one family of algorithms making use of
discontinuous coarse spaces but suitable for cell centered
finite volumes discretizations. 

In~\S\ref{sect:MotivationDCS}, we briefly recall the motivations
behind the use of discontinuous coarse space. In~\S\ref{sect:DCS-RJMin}, we
present the DCS-RJMin
algorithm. In~\S\ref{sect:ConvergenceTheoremDCSRJMin}, we prove that
under some conditions on the algorithm parameter, the $L^2$-norm of
the difference between two consecutive iterates goes to zero.
Finally, we present numerical results in~\S\ref{sect:NumericalResults}.

\section{Optimized Schwarz and Discontinuous Coarse Spaces}\label{sect:MotivationDCS}
Let's consider a polygonal domain $\Omega$ in $\mathbb{R}^2$. As a simple test case, we wish to solve
\begin{align*}
\eta u -\Lapl u&=f\;\text{in $\Omega$},\\ 
u&=0\;\text{on $\partial\Omega$}.
\end{align*} 
Without a coarse space, the Optimized Schwarz Method is defined as
\begin{algorithm}[Coarseless OSM]
\begin{enumerate}
\item Set $u_i^0$ to either the null function or to the coarse solution.
\item Until convergence
\begin{enumerate}
\item Set $u_i^{n+1}$ as the unique solution to
\begin{align*}
\eta u_i^{n+1} -\Lapl u_i^{n+1}&=f\;\text{in $\Omega_i$},\\ 
\frac{\partial u_i^{n+1}}{\partial \bm{n}_i} + p u_i^{n+1}&=
\frac{\partial u_j^{n}}{\partial \bm{n}_i} + p u_j^{n}
\;\text{on $\partial\Omega_i\cap\partial\Omega_j$},\\ 
u^{n+1}&=0\;\text{on $\partial\Omega_i\cap\partial\Omega$}.
\end{align*}
\end{enumerate}
\end{enumerate}
\end{algorithm}
In practical applications, such an algorithm should be accelerated
using Krylov methods. However, studying the iterative (Richardson) version can give
mathematical insight on the convergence speed of the Krylov accelerated algorithm.

The main shortcoming of the coarseless Optimized
Schwarz methods is the absence of direct
communication between distant subdomains. To get a scalable algorithm,
one can use a coarse space. A general version of a coarse space method
for the OSM is 
\begin{algorithm}[Generic OSM with coarse space]
\begin{enumerate}
\item Set $u_i^0$ to either the null function or to the coarse solution.
\item Until convergence
\begin{enumerate}
\item Set $u_i^{n+1}$ as the unique solution to
\begin{align*}
\eta u_i^{n+1/2} -\Lapl u_i^{n+1}&=f\;\text{in $\Omega_i$},\\ 
\frac{\partial u_i^{n+1/2}}{\partial \bm{n}_i} + p u_i^{n+1/2}&=
\frac{\partial u_j^{n}}{\partial \bm{n}_i} + p u_j^{n}
\;\text{on $\partial\Omega_i\cap\partial\Omega_j$},\\ 
u^{n+1/2}&=0\;\text{on $\partial\Omega_i\cap\partial\Omega$}.
\end{align*}
\item Compute in some way a coarse corrector $U^{n+1}$ belonging to
  the coarse space $X$, then set
\begin{equation*}
u^{n+1}=u^{n+1/2}+U^{n+1}.
\end{equation*}
\end{enumerate}
\end{enumerate}
\end{algorithm}
More important than the algorithm used to compute the coarse
correction $U^{n+1}$ is the choice of an adequate
coarse space itself. The ideas presented
in~\cite{Gander.Halpern.Santugini:2013:DD21-DCSDMNV} still apply. In particular,
the coarse space should contain discontinuous functions and the 
discontinuities of the coarse corrector should be located at the
interfaces between subdomains. For these reasons, we suppose the whole
domain $\Omega$ is meshed by either a coarse
triangular mesh or a cartesian mesh$\mathcal{T}_H$ and we use 
each coarse cell of $\mathcal{T}_H$ as a subdomain $\Omega_i$ 
of $\Omega$. The optimal theoretical coarse space $\mathcal{A}$ is the set of all
functions that are solutions to the homogenous equation inside each
subdomain: for linear problems, the errors made by any iterate are guaranteed to belong to
that space.  With an adequate algorithm to compute $U^{n+1}$, the coarse space $\mathcal{A}$
gives a convergence in a single coarse iteration. Unfortunately this
complete coarse space is only practical for one dimensional
problems as it is of infinite dimension in higher dimensions. One
should therefore choose a finite dimensional subset $X_d$ of
$\mathcal{A}$. 

The choice of the coarse space $X_d$ is primordial. It should have a dimension that
is a small multiple of the number of subdomains. To choose $X_d$, one
only need to choose boundary
conditions on every subdomain, then fill the interior of each subdomain
 by solving the homogenous equation in each subdomain. In this paper, we have not tried
 to optimize $X_d$ and for the sake of simplicity have chosen
$X_d$ as the set of all functions in $\mathcal{A}$ with linear Dirichlet boundary conditions
on each interface between any two adjacent subdomains.

\section{The DCS-RJMin Algorithm}\label{sect:DCS-RJMin}
We now describe the DCS-Robin Jump Minimizer algorithm:
\begin{algorithm}[DCS-RJMin]\label{algo:DCS-RJMin}
\item Set $p>0$ and $q>0$ and $X_d$ a finite dimensional subspace of 
$\mathcal{A}$.
\item Set $u^0$ to either $0$ or to the coarse space solution.
\item Until Convergence
\begin{enumerate} 
\item Set $u^{n+\frac{1}{2}}$ as the unique solution to
\begin{align*}
\eta u^{n+\frac{1}{2}}-\Lapl u^{n+\frac{1}{2}}&=f\;\text{in $\Omega_i$},\\
 \frac{\partial u_i^{n+\frac{1}{2}}}{\partial\bm{\nu}_{ij}}+pu_i^{n+\frac{1}{2}}&=
 \frac{\partial u_j^{n}}{\partial\bm{\nu}_{ij}}+pu_j^{n}\;\text{on $\partial\Omega_i\cap\partial\Omega_j$},\\
 u_i&=0\;\text{on $\partial\Omega_i\cap\partial\Omega_j$.}
\end{align*}
\item Set $U^{n+1}$ in $X_d$ as the unique coarse function that minimizes
\begin{gather*}
\begin{aligned}
\sum_{i=1}^N\sum_{j\in\mathcal{N}(i)}&\bigg\lVert
\frac{\partial(u_i^{n+\frac{1}{2}}+U_i^{n+1})}{\partial\bm{\nu}_{i}}+q(u_i^{n+\frac{1}{2}}+U_i^{n+1})\\
&-\frac{\partial(u_j^{n+\frac{1}{2}}+U_j^{n+1})}{\partial\bm{\nu}_{i}}-q(u_j^{n+\frac{1}{2}}+U_j^{n+1})
\bigg\rVert_{L^2(\partial\Omega_i\cap\partial\Omega_j)}^2,
\end{aligned}
\end{gather*}
where $\bm{\nu}_i$ is the outward normal to subdomain $\Omega_i$ and
$\mathcal{N}(i)$ the set of all $j$ such that $\Omega_j$ and
$\Omega_i$ are adjacent.
\item Set $u^{n+1}:=u^{n+1/2}+U^{n+1}$.
\end{enumerate}
\end{algorithm}


\section{Partial ``Convergence'' results for DCS-RJMin}\label{sect:ConvergenceTheoremDCSRJMin}
We don't have a complete convergence theorem for the DCS-RJMin
algorithm. However, we can prove the iterates of 
the DCS-RJMin algorithm are close to converging when $p=q$:
\begin{proposition}\label{prop:Not Really Convergence}
If $q=p$. Then, the
iterates produced by the DCS-RJMin algorithm~\ref{algo:DCS-RJMin}
satisfy $\lim_{n\to+\infty}\lVert u_i^{n+1/2}-u_i^n\rVert_{L^2}=0$.
\end{proposition}
\begin{proof}
Let $u$ be the mono-domain solution, set $e_i^n=u_i^n-u_i$, then, 
following Lions energy estimates~\cite{Lions:1990:SAM},
\begin{equation*}
\begin{split}
&\phantom{=}\eta \int_{\Omega_i}\lvert e_i^{n+1/2}-e_i^n\rvert^2\D\bm{x}
+\int_{\Omega_i}\lvert\nabla (e_i^{n+1/2}-e_i^n)\rvert^2\D\bm{x}\\
&=\int_{\partial\Omega_i}\frac{\partial(e_i^{n+1/2}-e_i^n
 )}{\partial\bm{\nu}}\cdot (e_i^{n+1/2}-e_i^n)\\
&=\frac{1}{4p}\left(
\int_{\partial\Omega_i}
\lvert\frac{\partial (e_i^{n+1/2}-e_i^n )}{\partial\bm{\nu}}
+p (e_i^{n+1/2}-e_i^n)\rvert^2
-\lvert\frac{\partial (e_i^{n+1/2}-e_i^n )}{\partial\bm{\nu}}
-p (e_i^{n+1/2}-e_i^n)\rvert^2\right)\\
&=\frac{1}{4p}\left(
\int_{\partial\Omega_i}
\lvert\frac{\partial (e_i^{n+1/2}-e_i^n )}{\partial\bm{\nu}}
+p (e_i^{n+1/2}-e_i^n)\rvert^2
 -\int_{\partial\Omega_i}\lvert\frac{\partial (e_i^{n+1/2}-e_i^n )}{\partial\bm{\nu}}
 -p (e_i^{n+1/2}-e_i^n)\rvert^2\right)\\
&=\frac{1}{4p}\biggr(
\sum_j\int_{\partial\Omega_i\cap\partial\Omega_j}
\left\lvert \left(\frac{\partial e_j^{n}}{\partial\bm{\nu}_i}+p e_j^n\right)-\left(\frac{\partial e_i^{n)}}{\partial\bm{\nu}_i}+p e_i^n\right)\right\rvert^2
\\ & \phantom{=}
-\sum_j\int_{\partial\Omega_i\cap\partial\Omega_j}
\left\lvert\left(\frac{\partial e_i^{n+1/2}}{\partial\bm{\nu}_i}-p e_i^{n+1/2}\right)-\left(\frac{\partial e_j^{n+1/2)}}{\partial\bm{\nu}_i}-p e_j^{n+1/2}\right)\right\rvert^2\biggr)
\end{split}
\end{equation*}
We sum the above equality over all subdomains $\Omega_i$ and get
\begin{multline*}
\eta \sum_i\int_{\Omega_i}\lvert e_i^{n+1/2}-e_i^n\rvert^2\D\bm{x}
+\int_{\Omega_i}\lvert\nabla (e_i^{n+1/2}-e_i^n)\rvert^2\D\bm{x}
=\\=
\sum_{(i,j)}\frac{1}{4p}\left(\int_{\Gamma_{ij}}\left\lvert\left\lbrack \frac{\partial e^{n}}{\partial\bm{\nu}_i}+p e^{n}\right\rbrack\right\rvert^2
-\int_{\Gamma_{ij}}\left\lvert\left\lbrack \frac{\partial e^{n+1/2}}{\partial\bm{\nu}_i}+p e^{n+1/2}\right\rbrack\right\rvert^2\right),
\end{multline*}
where $\lbrack\cdot\rbrack$ represents a jump across the interface.
Since the coarse step of the DCS-RJMin algorithm minimizes the
Robin Jumps, we have
\begin{multline*}
\eta \sum_i\int_{\Omega_i}\lvert e_i^{n+1/2}-e_i^n\rvert^2\D\bm{x}
+\int_{\Omega_i}\lvert\nabla (e_i^{n+1/2}-e_i^n)\rvert^2\D\bm{x}
\leq\\\leq
\sum_{(i,j)}\frac{1}{4p}\left(\int_{\Gamma_{ij}}\left\lvert\left\lbrack \frac{\partial e^{n}}{\partial\bm{\nu}_i}+p e^{n}\right\rbrack\right\rvert^2
-\int_{\Gamma_{ij}}\left\lvert\left\lbrack \frac{\partial e^{n+1}}{\partial\bm{\nu}_i}+p e^{n+1}\right\rbrack\right\rvert^2\right).
\end{multline*}
Summing over $n\geq0$ yields the stated result. 
\end{proof}


\section{Numerical Results}\label{sect:NumericalResults}
\begin{figure}[hpbt]
\begin{minipage}{\textwidth}
\begin{minipage}{0.5\textwidth}
\includegraphics[width=\textwidth]{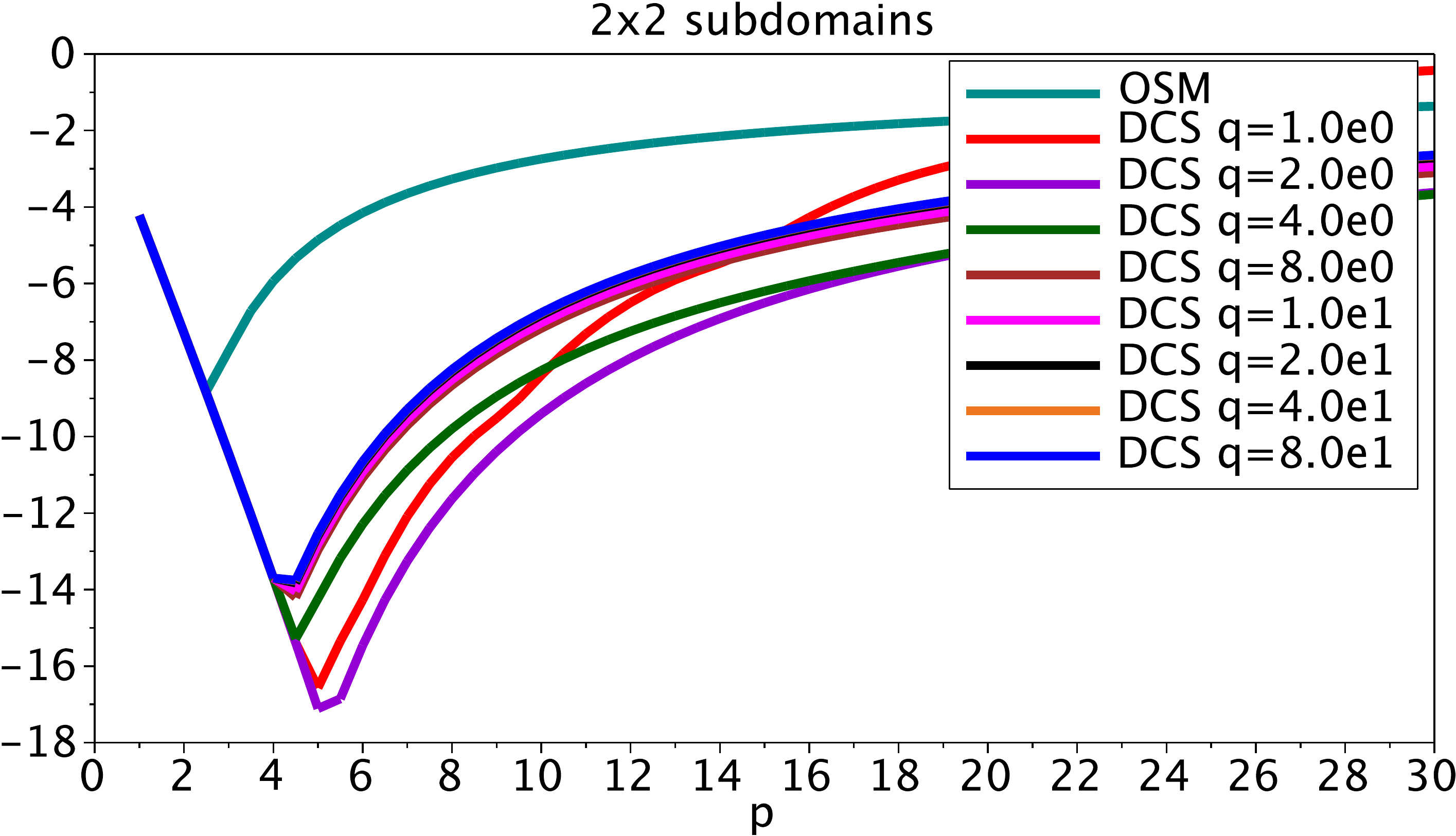}\\
\includegraphics[width=\textwidth]{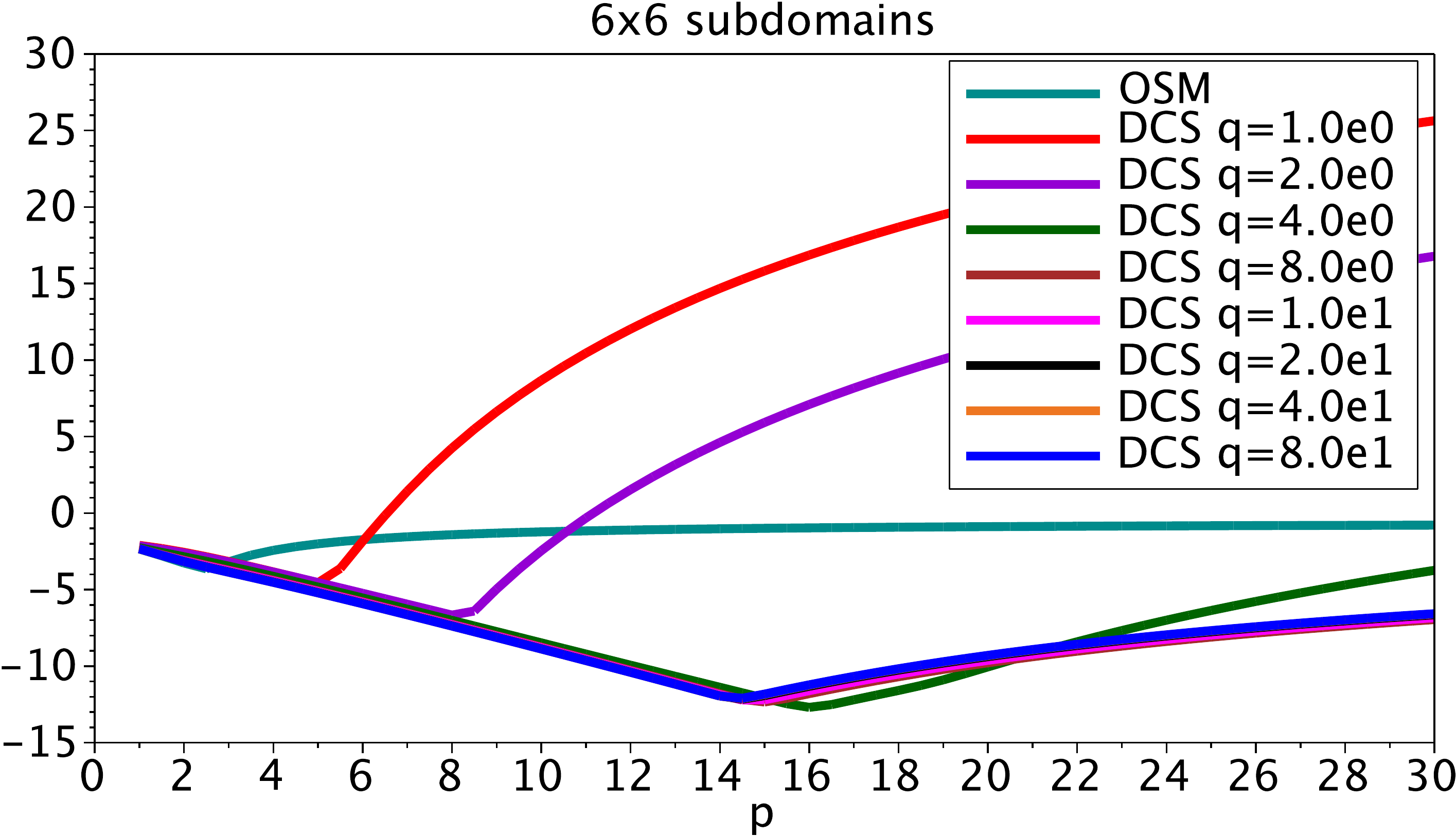}\\
\end{minipage}
\begin{minipage}{0.5\textwidth}
\includegraphics[width=1\textwidth]{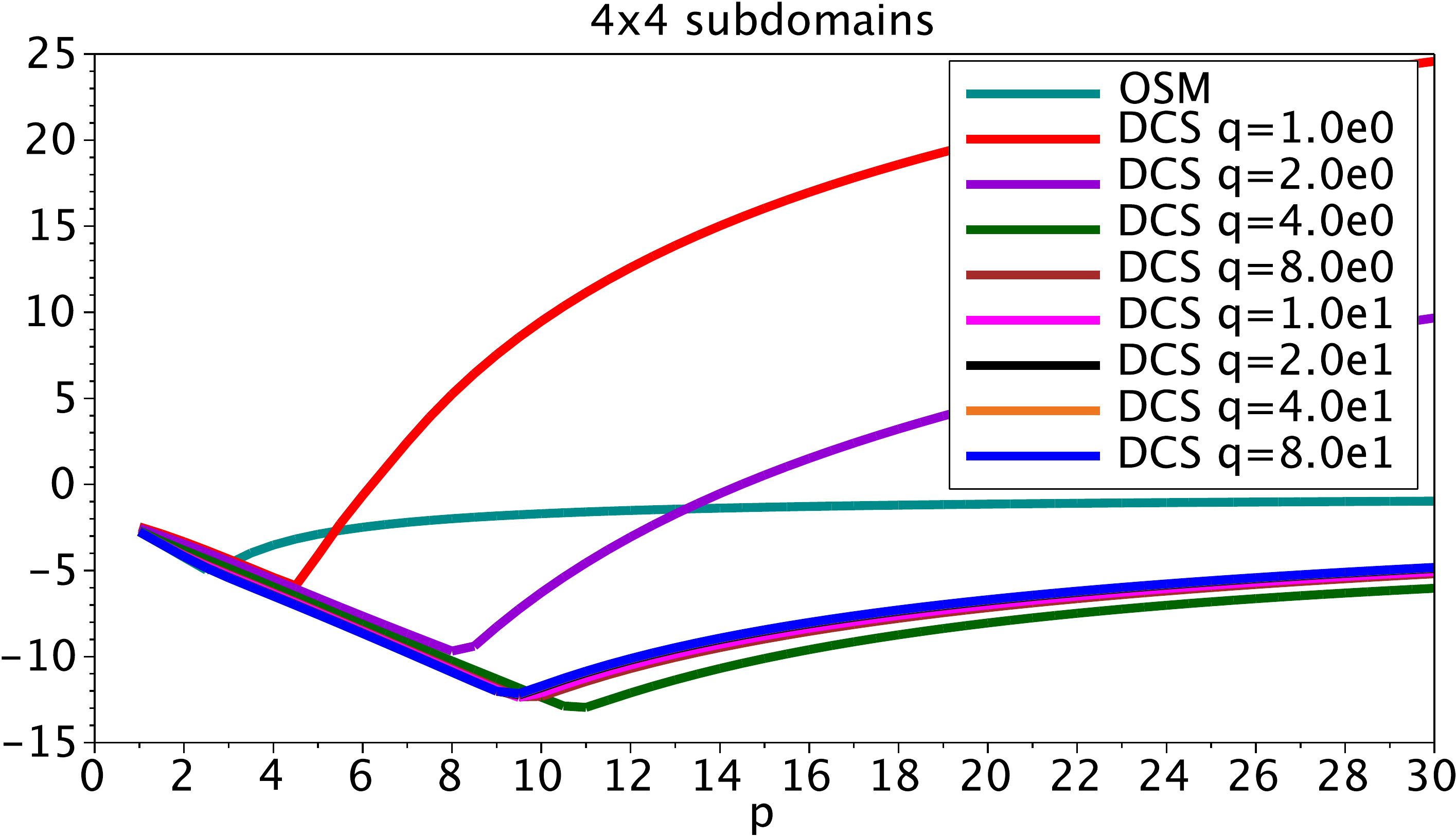}\\
\includegraphics[width=1\textwidth]{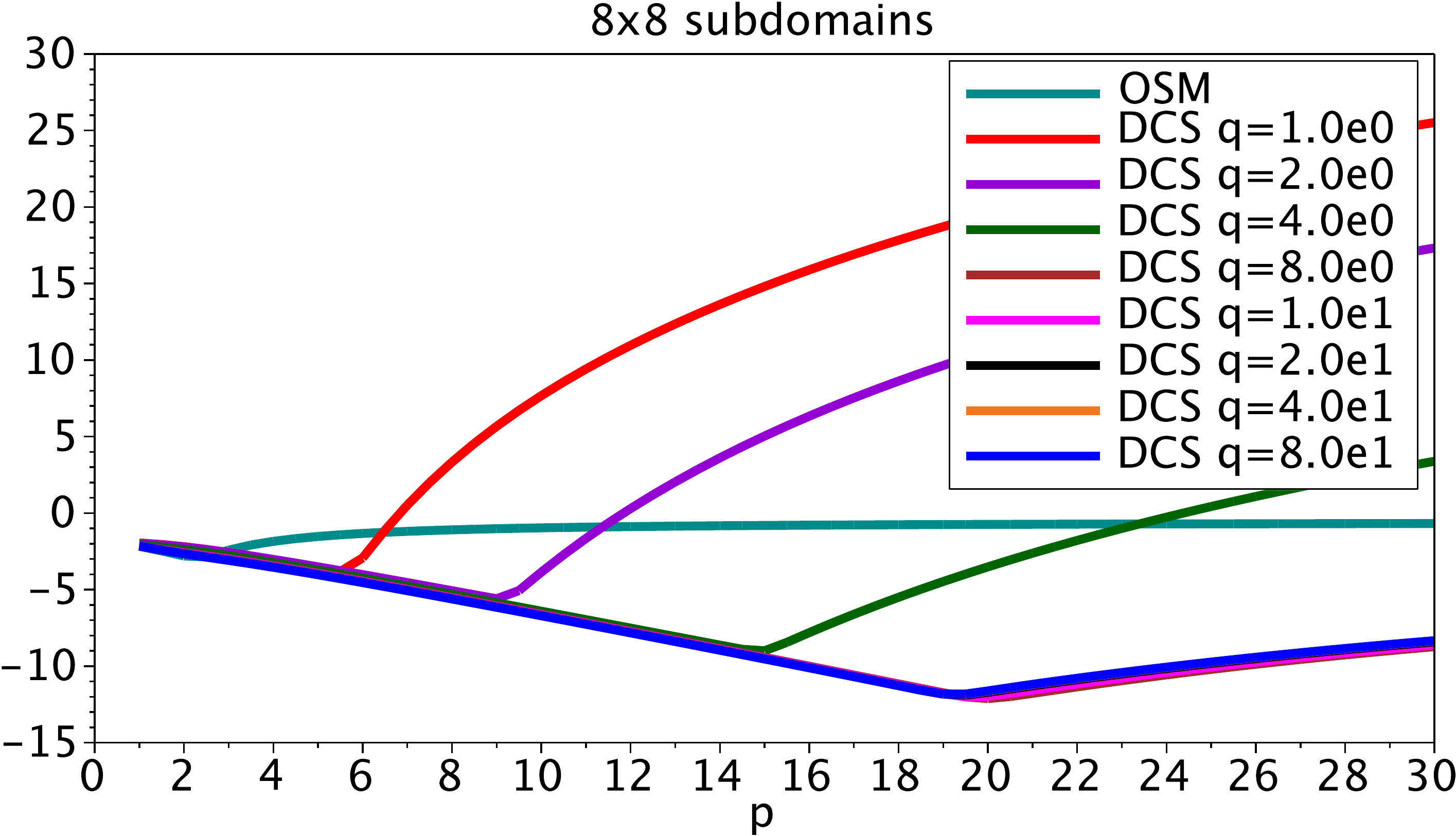}\\
\end{minipage}
\end{minipage}
\caption{
Convergence for OSM and DCS-RJmin with
  $\Omega=\lbrack0,4\rbrack^2$, $f(x,y)=0$ and random initial boundary
  conditions. Plotting $\log(\lVert e_{50}\rVert_\infty/\lVert
e_{0}\rVert_\infty)$.}\label{fig:MultiplePQ}
\end{figure}

We have implemented the DCS-RJMin algorithm in C++ for cell-centered
finite volumes on a cartesian grid. We chose
$\Omega=\rbrack0,4\lbrack\times\rbrack0,4\lbrack$, $\eta=0$ and
iterated directly on the errors by choosing $f=0$. We initialized 
the Robin boundary conditions at the interfaces between subdomains at
random and performed multiple runs of the DCS-RJMin algorithm for
various values of $p$, $q$ and of the number of subdomains. We had
$p$ vary from $1.0$ to $20.0$ with $0.5$ increments and $q$ takes the
following values $p_m\times10^{p_e}$ with $p_m$ in $\{1.0, 2.0, 4.0,
8.0\}$ and $p_e$ in $\{0,1\}$. We consider $2\times2$,
$4\times 4$, $6\times6$ and $8\times8$ subdomains. There are always $20\times20$
cells per subdomains. In Figure~\ref{fig:MultiplePQ},
we plot $\log(\lVert e_{50}\rVert_\infty/\lVert
e_{0}\rVert_\infty)$ as a function of $p$ for various values of
$q$. First, we notice that for each value of
$q$, the convergence deteriorates above a certain $p_q$.
In fact, for low values of $q$ and high values of $p$, the iterates
diverge. For two different values of $q$, the curves are very close when $p$ is
smaller than both $p_q$. We also notice than even though we could only
prove Proposition~\ref{prop:Not Really Convergence} for the case
$p=q$, we observe numerical convergence even when $p\neq q$. In fact
$p=q$ is not the numerical optimum. This is to be expected at the
intuitive level: for a theoretical proof of convergence, we want the algorithm to
keep lowering some functional. The existence of such a functional is
likely only if all the substeps
of the algorithm are optimized for the same kind of errors. If $p=q$, both the coarse step or the local step
will either remove low frequency errors (small $p$ and $q$) or high frequency
ones (high $p$ and $q$). An efficient numerical algorithm should have
substeps optimized for completely different kind of errors. This is
why efficient numerical algorithms are usually the ones for which the
convergence proofs are the more difficult.

\section{Conclusion}
In this paper, we have introduced a new discontinuous coarse space algorithm, the
DCS-RJMin, that is suitable for cell-centered finite volume
discretizations. The coarse space greatly improve numerical convergence.
It would be of great interest to study which is the optimal low-dimensional 
subspace of all piecewise discontinuous piecewise harmonic functions. 
Future work also includes the development of a possible
alternative to coarse space in order to get scalability: ``Piecewise
Krylov Methods'' where the same minimization problem than the one used
in DCS-RJMin is used but where the coarse space are made of
piecewise, per subdomain, differences between consecutive one-level iterates.
\bibliographystyle{plain}
\bibliography{ajoutmaths,ddm,info}

\end{document}